\documentclass{amsart}
\usepackage{amsmath}%用于Align环境
\usepackage{amsfonts}%用于\mathbb{}
\usepackage{amsthm}%用于定理环境
\usepackage{amssymb}%第一次用于写右箭头\nrightarrow 命令
\usepackage{enumerate}%计时器，用于编号公式之类;还用于列表环境。
\usepackage{cite}
\theoremstyle{plain}
\newtheorem{thm}{Theorem}[section]%写定理

\newtheorem{proposition}[thm]{Proposition}

\newtheorem{cor}[thm]{Corollary}

\newtheorem{lemma}[thm]{Lemma}
\theoremstyle{definition}%定义格式
\newtheorem{defn}[thm]{Definition}
\theoremstyle{remark}%注释格式
\newtheorem{rmk}[thm]{Remark}

\makeatletter%以下四行是用于写大小写罗马数字的，本文暂时只用得到大写。

\newcommand{\Rmnum}[1]{\expandafter\@slowromancap\romannumeral #1@}
\makeatother

\thanks{The first author is supported by NSF of China No.11925107, No.11671057 and No.11688101. The third author is supported by NSF of China No.11901090.}


\begin{document}

\title{The Gehring-Hayman type theorems on complex domains}
\author{Jinsong Liu\textsuperscript{1,2} $\&$ Hongyu Wang\textsuperscript{1,2} $\&$ Qingshan Zhou\textsuperscript{3}}
\address{$1.$ HLM, Academy of Mathematics and Systems Science,
Chinese Academy of Sciences, Beijing, 100190, China}
\address{$2.$ School of
Mathematical Sciences, University of Chinese Academy of Sciences,
Beijing, 100049, China }
\address{$3.$ School of Mathematics and Big Data, Foshan university, Foshan, Guangdong, 528000, China}
\email{liujsong@math.ac.cn, wanghongyu16@mails.ucas.ac.cn, q476308142@qq.com}

\begin{abstract} In this paper we establish {\it Gehring-Hayman} type theorems for some complex domains.
Suppose that $\Omega\subset \mathbb{C}^n$ is a bounded $m$-convex domain with Dini-smooth boundary, or a bounded strongly pseudoconvex domain with $C^2$-smooth boundary. Then we prove that the Euclidean length of Kobayashi geodesic $[x,y]$ in $\Omega$ is less than $c_1|x-y|^{c_2}$. Furthermore, if $\Omega$ endowed with the Kobayashi metric is Gromov hyperbolic, then we can generalize this result to quasi-geodesics with respect to Bergman metric, Carath\'{e}odory metric or K\"{a}hler-Einstein metric.

As applications, we prove the bi-H\"{o}lder equivalence between the Euclidean boundary and the Gromov boundary. Moreover, by using this boundary correspondence, we can show some extension results for biholomorphisms, and more general rough quasi-isometries with respect to the Kobayashi metrics between the domains.
\end{abstract}

\maketitle

\section{\noindent{{\bf Introduction}}}\label{In}
Given any two points $x,\: y\in B(0,1)\subset\mathbb{C}$, the hyperbolic geodesic $[x,y]$ is the arc of the circle through $x$ to $y$ orthogonal to the boundary $\partial B(0,1)$. Therefore, the Euclidean length of $[x,y]$ satisfies
$$l_{d}([x,y])\leq \frac{\pi}{2}|x-y|.$$
This simple fact is an instance of the following famous theorem due to Gehring-Hayman\cite{Gehring1962An}.
\begin{thm}
If $\Omega$ is a simply connected planar domain $(\Omega\neq\mathbb{C})$, then there exists $C>0$ such that, for any $x,\:y\in\Omega$,
$$l_{d}([x,y])\leq C \:l_d(\gamma),$$
where $[x,y]$ is the hyperbolic geodesic joining $x$ and $y$, and $\gamma\subset \Omega$ is any curve with end points $x$ and $y$,
and $l_{d}$ denotes the Euclidean length.\end{thm}
%in a simply connected plane domain
% the hyperbolic geodesic minimizes the euclidean length among all
%curves in the domain with the same end points up to a universal multiplicative constants.
The Gehring-Hayman theorem for hyperbolic geodesics in multiply connected plane
domains has been studied by Pommerenke\cite{Pommerenke1979Uniformly}.
In \cite{Gehring1979Uniform}, Gehring and Osgood generalized the Gehring-Hayman theorem to {\it quasihyperbolic geodesics} on {\it uniform domains} in $\mathbb{R}^n$. Moreover, this theorem has been generalized to domains in $\mathbb{R}^n$ that are quasiconformally equivalent to uniform domains\cite{Heinonen1994Quasiconformal}. Subsequently, this type of Gehring-Hayman theorem has been established in many different situations and plays an important role in modern function theory and quasiconformal analysis, such as \cite{Bonk2001Uniformizing, BKR, Heinonen1994Quasiconformal, Her06, KL}. For instance, Bonk, Heinonen and Koskela \cite{Bonk2001Uniformizing} observed that the Gehring-Hayman property and the {\it Separation property} could be used to characterize the {\it Gromov hyperbolicity} of domains in $\mathbb{R}^n$. This conjecture has been verified by Balogh and Buckley in \cite{BaloghGeometric}.

In \cite{balogh2000gromov} Balogh and Bonk investigated the Gromov hyperbolicity of bounded {\it strictly pseudoconvex domains} in $\mathbb{C}^n$ equipped with the Kobayashi metrics. Recently, Zimmer \cite{Zimmer2016Gromov} discussed the Gromov hyperbolicity of  bounded convex domains of {\it finite type} endowed with the Kobayashi metrics. Motivated by these results, in this paper we will study the geometric properties of Kobayashi geodesics in those strongly pseudoconvex and $m$-convex domains. In particular, we will investigate the relationship between the Gehring-Hayman property, the Separation property and several hyperbolic type metrics (in the sense of Gromov) such as Kobayashi metric, Bergman metric, Carath\'{e}odory metric and K\"{a}hler-Einstein metric.

\bigskip
In what follows, we will use quite a few constants whose precise values usually does not matter, unless stated otherwise.

We first prove some results similar to the classical Gehring-Hayman theorem for Kobayashi geodesics in {\it $m$-convex} domains or {\it strongly pseudoconvex} domains. We refer the reader to Section 2 for the precise definitions.
\begin{thm}\label{thm}
Let $\Omega$ be a bounded $m$-convex domain in $\mathbb{C}^n (n\geq 2)$ with Dini-smooth boundary. Then for any $0<c_2<1/(12m^2-8m)$, there exists a constant $c_1>0$ such that, for any $x,\: y\in\Omega$,
$$l_{d}([x,y])\leq c_1|x-y|^{c_2},$$
where $[x,y]$ is any Kobayashi geodesic joining $x$ and $y$ in $\Omega$.

If, in addition, $(\Omega, \:K_{\Omega})$ is Gromov hyperbolic and $\gamma$ is a Kobayashi $\lambda$-quasi-geodesic connecting $x$ and $y$ with $\lambda\geq 1$, then there exists a constant $c_1'>0$ such that
$$l_{d}(\gamma)\leq c_1'|x-y|^{c_2}.$$
\end{thm}
%\begin{thm}
%Let $\Omega$ be a bounded $m$-convex domain in $\mathbb{C}^n(n\geq2)$ with $C^2$ smooth boundary, then for any $c_2<\frac{1}{12m^2-8m}$, there exist $c_1>0$ such that $\forall x,y\in\Omega$,
%$$l_{d}([x,y])\leq c_1|x-y|^{c_2}$$
%where $[x,y]$ is the geodesic joining $x$ and $y$ in the Kobayashi metric and $l_d$ is the Euclidean length.
%\end{thm}
Recently Zimmer showed the following theorem.
\begin{thm}[Corollary 7.2,\cite{Subelliptic}]\label{thm-Z}
If $\Omega$ is a bounded convex domain in $\mathbb{C}^n (n\geq 2)$ and $(\Omega,K_{\Omega})$ is Gromov hyperbolic, then $\Omega$ is $m$-convex for some $m\geq1$.
\end{thm}
The following result is a direct consequence of Theorem \ref{thm} and Theorem \ref{thm-Z}.
\begin{cor}\label{z1}
Suppose that $\Omega$ is a bounded convex domain in $\mathbb{C}^n (n\geq 2)$ with Dini-smooth boundary and that $(\Omega,\: K_{\Omega})$ is Gromov hyperbolic, where $K_\Omega$ is the Kobayashi metric of $\Omega$. Then for any $\lambda\geq 1$, there exist constants $c_1,c_2>0$ such that for all $x,y\in\Omega$,
$$l_{d}(\gamma)\leq c_1|x-y|^{c_2},$$
where $\gamma$ is a $\lambda$-quasi-geodesic in the Kobayashi metric with end points $x$ and $y$.
\end{cor}
These results show that the Kobayashi geodesics (or quasi-geodesics) are essentially also short in the Euclidean sense.

\bigskip
The proof of Theorem \ref{thm} requires the following lemma. In what follows, denote $a\vee b:= \max\{a, b\}$ and $a\wedge b:= \min\{a, b\}$ for $a,\:b \in \mathbb R$.
\begin{lemma}\label{est14}
Let $\Omega$ be a bounded $m$-convex domain in $\mathbb{C}^n (n\geq 2)$ with Dini-smooth boundary, and let $[x,y]\subset \Omega$ be a Kobayashi geodesic joining $x$ and $y$. Then for any $\alpha>3m^2-2m$, there exists a constant $\tilde{C}>0$ such that, for every $\omega\in [x,y]$,
\begin{align}
\delta_{\Omega}(\omega)\geq \tilde{C}\left(l_{d}([x,\omega])\wedge l_{d}([\omega,y])\right)^{\alpha},
\end{align}
where $\delta_{\Omega}(\omega)$ is the Euclidean distance from $\omega$ to $\partial\Omega$.
\end{lemma}

Note that we may compare Lemma \ref{est14} with the {\it Separation property}, which states that whenever $[x,y]$ is a geodesic in $(\Omega, \:K_{\Omega})$, $z\in[x,y]$ and $\gamma$ is a curve in $\Omega$ connecting the subcurves $[x,z]$ and $[z,y]$ of $[x,y]$, then for some $a>0$,
$$B(z,a\delta_{\Omega}(z))\cap\gamma\neq\emptyset.$$
In fact Lemma \ref{est14} gives that $$B(z,\delta_{\Omega}^{\frac{1}{\alpha}}(z)/\tilde{C})\cap\gamma\neq\emptyset.$$
We refer the reader to \cite{BaloghGeometric,Bonk2001Uniformizing,KLM} for more information on the Separation property.

\bigskip
Next we prove a similar result for strongly pseudoconvex domains as follows.
 \begin{thm}\label{thm-2}
Let $\Omega$ be a bounded strongly pseudoconvex domain in $\mathbb{C}^n (n\geq 2)$ with $C^2$-smooth boundary, and let $\gamma$ be a Kobayashi $\lambda$-quasi-geodesic in $\Omega$ joining $x$ and $y$ with $\lambda\geq 1$. Then for any $0<c_2<1/16$, there exists a constant $c_1>0$ such that
$$l_{d}(\gamma)\leq c_1|x-y|^{c_2}.$$
\end{thm}
\begin{rmk}
It should be pointed out that a bounded strongly pseudoconvex domain is not necessarily convex. Note that in Theorem \ref{thm-2} we still get the result with $|x-y|$ instead of the inner distance of $\Omega$ between $x$ and $y$. That's due to the smoothness assumption of its boundary. In fact by Corollary 8 in \cite{Nikolov2015Estimates}, we know that a bounded domain with Dini-smooth (in particular, $C^2$-smooth) boundary is always uniform (see Section \ref{km} for the precise definitions).
\end{rmk}

Note that the Kobayashi metric, the Bergman metric, the Carath\'{e}odory metric, and the K\"{a}hler-Einstein metric on a bounded convex domain are bilipschitzly equivalent to each other. This is due to the fact that bounded convex domains and strongly pseudoconvex domains are both uniformly squeezing (see Section \ref{usq} for the precise definition).
Denoting by $\varrho_{\Omega}$ one of the above metrics on $\Omega$, then we have:
\begin{cor}\label{cor}
Suppose that $\Omega$ is a bounded domain in $\mathbb{C}^n (n\geq 2)$ and that $\Omega$ satisfies either\\
$($a$)$ $\Omega$ is convex with Dini-smooth boundary and $(\Omega,\varrho_{\Omega})$ is Gromov hyperbolic; or \\
$($b$)$ $\Omega$ is strongly pseudoconvex with $C^2$-smooth boundary. \\
Then there exist $c_1, \:c_2>0$ such that, $\forall x,y\in\Omega$,
$$l_{d}(\gamma)\leq c_1|x-y|^{c_2},$$
where $\gamma$ is a $\lambda$-quasi-geodesic from $x$ to $y$ with respect to the metric $\varrho_{\Omega}$.
\end{cor}

As applications, we will use the Gehring-Hayman type theorem and the Separation property to investigate the bi-H\"{o}lder homeomorphism between the Euclidean closure and the Gromov closure of certain complex domains.
Our motivation arises from the following result due to Balogh and Bonk (refer to Section \ref{appp} for the precise definitions).

%\begin{ques}\label{q} $($Question 1.3,\cite{Homeomorphic}$)$ For which classes of domains $\Omega\subset \mathbb{C}^n$ is the following true: if the Kobayashi metric $K_\Omega$ on $\Omega$ is Cauchy complete and Gromov hyperbolic, then the identity map $id:\Omega\to \Omega$ extends to a homeomorphism
%$\overline{id}:\overline{\Omega}^*\to\overline{\Omega}^G$,
%where $\overline{\Omega}^*$ denotes the Euclidean end compactification of $\Omega$ and $\overline{\Omega}^G$ is the Gromov compactification of the metric space $(\Omega,K_\Omega)$?
%\end{ques}

%In dimension one, the Carath\'eodory extension theorem gives an affirmative answer to Question \ref{q} for simply connected domains different from $\mathbb{C}$.
 %In \cite {balogh2000gromov}, Balogh and Bonk proved a more precise result on strongly pseudoconvex domain:

\begin{thm}[Theorem 1.4,\cite{balogh2000gromov}]\label{cc}
Let $\Omega \subseteq \mathbb{C}^{n}, n \geq 2,$ be a bounded, strongly pseudoconvex domain with $C^{2}$-smooth boundary. If $K_\Omega$ is the Kobayashi metric on $\Omega,$ then the metric space $\left(\Omega, \: K_\Omega\right)$ is Gromov hyperbolic. The Gromov boundary $\partial_{G} \Omega$ of $\left(\Omega, K_\Omega\right)$ can be identified with the Euclidean boundary $\partial \Omega$. The Carnot-Carathéodory metric $d_{H}$ on $\partial \Omega$ lies in (and thus determines) the canonical class of snowflake equivalent metrics on $\partial_{G} \Omega$.
\end{thm}

For a bounded strongly pseudoconvex domain $\Omega\subset \mathbb{C}^n$ with $C^2$-smooth boundary,
since
$$C_{1}|p-q| \leq d_{H}(p, \:q) \leq C_{2}|p-q|^{1 / 2}, \quad \text {for} \:\: p, \:q \in \partial \Omega,$$
we deduce that the map $\partial_G \Omega \rightarrow \partial \Omega$ defined by Theorem \ref{cc} is bi-H\"older continuous (See \cite{balogh2000gromov} and the references given there for more information).
Thus the visual metric of $\partial_{G}\Omega$ and the Euclidean metric of $\partial\Omega$ are bi-H\"{o}lder
equivalent to each other.

%By using the materials from coarse geometry and the dynamical properties of commuting $1$-Lipschitz maps in Gromov hyperbolic spaces, Bracci, Gaussier and Zimmer positively answered this question on Gromov hyperbolic convex domains.
On the other hand, Bracci, Gaussier and Zimmer \cite{Homeomorphic} get the following result on convex domains.
\begin{thm}[Theorem 1.4, \cite{Homeomorphic}]\label{dd}
Let $\Omega$ be a $\mathbb{C}$-proper convex domain on $\mathbb{C}^{n}.$ If $\left(\Omega, \: K_{\Omega}\right)$ is Gromov hyperbolic, then the identity map id: $\Omega \rightarrow \Omega$ extends to a homeomorphism (to simplify notation, still use the same name) id: $\overline{\Omega}^{\star} \rightarrow \overline{\Omega}^{G}$.
where $\overline{\Omega}^{\star}$ denotes the Euclidean end compactification of $\Omega$ and $\overline{\Omega}^G$ is the Gromov compactification of the metric space $(\Omega,\:K_\Omega)$.
\end{thm}

Our main result in this direction is as follows.
\begin{thm}\label{app}
Suppose that $\Omega$ is a bounded domain in $\mathbb{C}^n (n\geq 2)$ and suppose that $\Omega$ satisfies either\\
$($a$)$ $\Omega$ is convex with Dini-smooth boundary and $(\Omega,\: K_{\Omega})$ is Gromov hyperbolic; or \\
$($b$)$ $\Omega$ is strongly pseudoconvex with $C^2$-smooth boundary. \\
Then the identity map $id: \Omega \rightarrow \Omega$ extends to a bi-H\"{o}lder homeomorphism (for simplicity of notation, use the same name)
\begin{align*}
id:(\partial\Omega,\: |\cdot|)\rightarrow (\partial_{G}\Omega,\: \rho_G)
\end{align*}
between the boundaries, where $\rho_G$ belongs to the visual metrics class on the Gromov boundary of $(\Omega, \: K_{\Omega})$.
\end{thm}

\begin{rmk}
\noindent
\begin{enumerate}
 % \item Theorem \ref{app} provides a partial answer to Question \ref{q} when $\Omega$ is convex with Dini-smooth boundary and $(\Omega,K_{\Omega})$ is Gromov hyperbolic or strongly pseudoconvex with $C^2$-smooth boundary. In fact, we show that the extended map is bi-H\"{o}lder continuous, and Theorem \ref{dd} only shows the homeomorphism equivalence.
 \item Although the assertion for Case (b) in Theorem \ref{app} follows easily from Theorem \ref{cc}, our approach is different with \cite{balogh2000gromov}. Our proof is based on the Gehring-Hayman type theorem and the Separation property.
 \item Gromov boundary equipped with any two visual metrics are power quasisymmetrically and so bi-H\"{o}lder equivalent to each other. Thus the boundary extension of the identity map in Theorem \ref{app} is bi-H\"{o}lder with respect to any visual metric on the Gromov boundary.
\end{enumerate}
\end{rmk}
The final goal of this paper is to apply this boundary correspondence to investigate boundary extension results for biholomorphisms, and more general rough quasi-isometries with respect to the Kobayashi metrics between the domains. In \cite{balogh2000gromov}, Balogh and Bonk generalized this kind of results for rough quasi-isometries in the Kobayashi metrics. In \cite{Homeomorphic}, Bracci, Gaussier and Zimmer proved the following result:
\begin{thm} Let $\Omega_1$ and $\Omega_2$ be domains in $\mathbb{C}^{n}$. We assume:
\begin{enumerate}
 \item $\Omega_1$ is either a bounded, $C^{2}$-smooth strongly pseudoconvex domain, or a convex $\mathbb{C}$-proper domain such that $\left(\Omega_1, K_{\Omega_1}\right)$ is Gromov hyperbolic,
 \item $\Omega_2$ is convex.
\end{enumerate}
Then every roughly quasi-isometric homeomorphism $F: \left(\Omega_1, \:K_{\Omega_1}\right) \rightarrow\left(\Omega_2, \: K_{\Omega_2}\right)$ extends to homeomorphism (use the same name) $F: \overline{\Omega}_1^{\star}\rightarrow \overline{\Omega}_2^{\star}$,
where $\overline{\Omega}_i^{\star}$ is the Euclidean end compactification of $\Omega_i, \: i=1, 2$.
\end{thm}

As a consequence of Theorem \ref{app}, we prove the following bi-H\"{o}lder homeomorphism extension result.
\begin{cor}\label{cor2}
For $i=1, 2$, suppose that $\Omega_i \subset \mathbb{C}^n ,\: n\geq 2$, are bounded, and suppose that $\Omega_i$ satisfy either \\
\indent($a$)$ \: \Omega_i$ is a convex domain with Dini-smooth boundary and $(\Omega_i, \:K_{\Omega_i})$ is Gromov hyperbolic; or\\
\indent($b$)$ \: \Omega_i$ is a strongly pseudoconvex domain with $C^2$-smooth boundary.\\
Let $f:\Omega_1\rightarrow\Omega_2$ be a homeomorphism that is a rough quasi-isometry with respect to the Kobayashi metrics $K_{\Omega_i}$. Then $f$ has a homeomorphic extension $\bar{f}: \overline{\Omega}_{1} \rightarrow \overline{\Omega}_{2}$ such that the induced boundary map $\left.\bar{f}\right|_{\partial \Omega_{1}}: \partial \Omega_{1} \rightarrow \partial \Omega_{2}$ is bi-H\"{o}lder with respect to the Euclidean metric.
\end{cor}
Note that every biholomorphism between the complex domains is an isometry with respect to the Kobayashi metrics. Therefore, Corollary \ref{cor2} clearly holds for biholomorphisms between the complex domains.

%Additionally, we remark that there is a lot of results concerning the continuous extension of biholomorphic mappings or proper holomorphic mappings, see \cite{forstneric1993proper}. In general if $\Omega_1, \: \Omega_2$ are bounded pseudoconvex domains with
%$$\frac{1}{C}\delta_{\Omega_1}^{\frac{1}{\nu_1}}(z)\leq\delta_{\Omega_2}(f(z))\leq C\delta_{\Omega_1}^{\nu_1}(z), \:\:\: \forall z\in\Omega_1,$$ and the Kobayashi metric
%$$k_{\Omega_2}(\omega,v)\geq\frac{C|v|}{{\delta_{\Omega_2}(\omega)}^{\nu_2}}, \:\:\:\: \forall \omega\in\Omega_2, \:\: v\in\mathbb{C}^n, $$
%for some $\nu_1, \nu_2, C>0$, then the proper holomorphic map $f:\Omega_1\rightarrow
%\Omega_2$ extends to a H\"{o}lder continuous map of
%$\overline{\Omega}_{1}$. This holds in particular if $\Omega_i$ are
%strongly pseudoconvex domains and more generally pseudoconvex
%domains with finite type.

\bigskip
The rest of this paper is organized as follows. In Section 2 we recall some necessary definitions and preliminary results. Section 3 is focus on the proofs of Theorem \ref{thm} and Lemma \ref{est14}. The proofs of Theorem \ref{thm-2} and Corollary \ref{cor} are presented in Section 4. At last we prove Theorem \ref{app} and Corollary \ref{cor2} in Section 5.
\bigskip
\section{\noindent{{\bf Preliminaries}}}
\subsection{Notation}\

(1) \:For $z\in\mathbb{C}^n$, let $|\cdot|$ and $d$ denote the standard Euclidean
norm, and let $|z_1-z_2 |$ and $d(z_1,z_2)$ be the standard Euclidean
distance of $z_1,z_2\in \mathbb{C}^n$.

(2) \:Given an open set $\Omega\subsetneq\mathbb{C}^n,\:x\in\Omega$ and
$v\in\mathbb{C}^n\backslash\{0\}$, denote
$$\delta_{\Omega}(x)=\inf\left\{d(x,\:\xi):\xi\in\partial \Omega\right\}$$
as before, and denote
$$\delta_{\Omega}(x,v)=\inf\{d(x,\:\xi):\xi\in\partial \Omega\cap(x+\mathbb{C}v)\}.$$

(3) \:For any curve $\sigma: [a,\:b]\rightarrow \Omega$, its
Euclidean length $l_d(\sigma)$ is defined by
$$
l_d(\sigma)=\sup\sum^n_{\nu=1}\left|\sigma(t_\nu)-\sigma(t_{\nu-1})\right|,
$$
where the supremum is taken over all possible partitions ${a=t_0\leq t_1\leq \cdots \leq t_n=b}$ of the interval $[a,\:b]$ and all $n\in \mathbb N$.

(4) \:For any $z_0 \in \mathbb C^n$ and $\epsilon >0$, we denote by $B_{\epsilon}(z_0)$ or $B(z_0,\epsilon)$ the open ball
$\{z\in \mathbb C^n| \:|z-z_0|<\epsilon\}$.
\subsection{$M$-convex domains and strongly pseudoconvex domains}\

In \cite{MercerComplex}, Mercer introduced the class of $m$-{\it convex} domains. Now we give the definition of $m$-convex domains as follows.
\begin{defn}\label{def2}
A bounded convex domain $\Omega\subset \mathbb{C}^n$ with $n\geq 2$ is called $m$-convex for some $m\geq 1$ if there exists $C>0$ such that
\begin{align}\label{def}
\delta_{\Omega}(z;v)\leq C\: \delta_{\Omega}^{\frac{1}{m}}(z)
\end{align} for any $z\in\Omega, \: v\in\mathbb{C}^n$.
\end{defn}
Note that the $m$-convexity is related to the finite type by the following proposition.
\begin{proposition}[\cite{Zimmer2016Gromov}, \:Proposition 9.1]
Given a bounded convex domain $\Omega \subset \mathbb C^n (n\geq 2)$ with smooth boundary, then $\Omega$ is $m$-convex for some $m\in\mathbb{N}$ if and only if $\partial\Omega$ has finite line type in the sense of D'Angelo.
\end{proposition}
\begin{defn}
A domain $\Omega=\{z|\rho(z)<0\}$ in $\mathbb{C}^n (n\geq 2)$ with $C^2$-smooth boundary is called {\it strongly pseudoconvex} if the Levi
form of the boundary
$$
L_{\rho}(p ; \: v)=\sum_{\nu, \mu=1}^{n} \frac{\partial^{2} \rho}{\partial z_{\nu} \partial \bar{z}_{\mu}}(p) v_{\nu} \bar{v}_{\mu}, \quad \text { for } v=\left(v_{1}, \ldots, v_{n}\right) \in \mathbb{C}^{n}
$$
is positive definite for every $p\in\partial\Omega$.
\end{defn}
It's well known that a strongly pseudoconvex domain is locally biholomorphic to a strongly convex domain. Thus strongly pseudoconvex domains have many properties similar to $2$-convex domains.
\subsection{Kobayashi metrics}\label{km}\

Given a domain $\Omega \subset \mathbb{C}^{n} (n\geq 2)$, the (infinitesimal)
Kobayashi metric is the pseudo-Finsler metric defined by
$$k_{\Omega}(x ; v)=\inf \left\{|\xi| : f \in \operatorname{Hol}(\mathbb{D}, \Omega), \:\text { with } f(0)=x,
d(f)_{0}(\xi)=v\right\}.$$ Define the Kobayashi length of any curve $\sigma:[a,b]\rightarrow \Omega$
to be
$$l_k(\sigma)=\int_{a}^{b} k_{\Omega}\left(\sigma(t) ; \sigma^{\prime}(t)\right) d
t.$$
It is a consequence of a result due to Venturini \cite{VenturiniPseudodistances}, which is based on an observation by Royden \cite{royden1971remarks}, that
the Kobayashi pseudo-distance can be given by:
\begin{align*}
K_{\Omega}(x, y)&=\inf_\sigma \big\{l_k(\sigma)| \:\sigma :[a, b]
\rightarrow \Omega \text { is any absolutely continuous curve }\\
& \text { with } \sigma(a)=x \text { and } \sigma(b)=y \big\}.
\end{align*}
%\begin{align*}
%K_{\Omega}(x, y)&=\inf \left{l_k(\sigma)| \:\sigma :[a, b]
%\rightarrow \Omega \text { is any absolutely continuous curve }\\
%& \text { with } \sigma(a)=x \text { and } \sigma(b)=y \right }.
%\end{align*}
There are some estimates concerning the Kobayashi metric on convex domains.
\begin{lemma}[\cite{Graham1975Boundary}]\label{est2}
If $\Omega \subset \mathbb{C}^{n}$ is a bounded convex domain, then for all $x \in \Omega$ and $v \in \mathbb C^n$,
\begin{align}\label{est5}
\frac{|v|}{2\delta_{\Omega}(x ; v)} \leq k_{\Omega}(x ; v) \leq \frac{|v|}{\delta_{\Omega}(x ; v)}.
\end{align}
%\begin{align}\label{est4}
%K_{\Omega}(x, y)\geq \frac{1}{2}\log \left(\frac{|x-\xi|}{|y-\xi|}\right),
%\end{align}
%where $\xi \in \partial\Omega \cap \{x+(x-y)\cdot\mathbb C\}$.
\end{lemma}
\begin{lemma}[Lemma 4.2,\cite{ZimmerCharacterizing}]\label{est25}
Suppose $\Omega \subset \mathrm{C}^{n}$ is a convex domain and $H \subset \mathbb{C}^{n}$ is a complex hyperplane
with $H \cap \Omega=\emptyset$. Then, for any $x, y \in \Omega$ we have
\begin{align}\label{est4}
K_{\Omega}\left(x, \: y\right) \geq \frac{1}{2}\left|\log \frac{d\left(x,H\right)}{d\left(y,H\right)}\right|.
\end{align}
\end{lemma}
%We also need an auxiliary result given in \cite{Nikolov2015The}.
%
%\begin{lemma}\label{est6}
%Suppose that $\Omega \subset \mathbb{C}^{n}$ is an bounded convex domain, then for all $x, \:y \in \Omega, \:v \in \mathbb C^n$,
%$$
%K_{\Omega}(x,y)\geq\frac{1}{2}\log\Big(1+\frac{|x-y|}{\delta_{\Omega}(x,v)\wedge \delta_{\Omega}(y,v)}\Big)
%$$
%\end{lemma}
Recall that a $C^1$-smooth boundary point $p$ of a domain $\Omega$ in $\mathbb{C}^n$ is said to be {\it Dini-smooth} (or {\it Lyapunov-Dini-smooth}), if the inner unit normal vector $\textbf{n}$ to $\partial\Omega$ near $p$ is a Dini-continuous function. This means that there exists a neighborhood $U$ of $p$ such that
$$\int^{1}_{0} \frac{\omega(t)}{t}dt < +\infty,$$
where
$$\omega(t) = \omega(\textbf{n},\partial\Omega \cap U, t) := \sup\left\{|\textbf{n}_x-\textbf{n}_y|:|x-y|<t, \:\: x, \: y \in\partial\Omega\cap U\right\}$$
is the respective modulus of continuity. Note that {\it Dini-smooth} is a weaker condition than $C^{1,\epsilon}$-smooth.

Here a Dini-smooth domain means that each boundary point of $\Omega$ is a Dini-smooth point. Then we have
\begin{lemma}[Corollary 8,\cite{Nikolov2015Estimates}]\label{nik}
Let $\Omega$ be a Dini-smooth bounded domain in $\mathbb{C}^n$ and $x,\: y\in \Omega$. Then there exists a constant $A>1+\sqrt{2}/2$ such that
\begin{align}\label{est}
K_{\Omega}(x,\: y)\leq \log\left(1+\frac{A|x-y|}{\sqrt{\delta_{\Omega}(x)\delta_{\Omega}(y)}}\right).
\end{align}
%Since when $A>1$ and $t>0$, $\log(1+At)\leq A\log(1+t)$, thus
%\begin{align}
%K_{\Omega}(x,y)\leq \log(1+\frac{A|x-y|}{\sqrt{\delta_{\Omega}(x)\delta_{\Omega}(y)}})\leq A\log(1+\frac{|x-y|}{\sqrt{\delta_{\Omega}(x)\delta_{\Omega}(y)}})\end{align}
\end{lemma}

\begin{proposition}[\cite{BarthConvex}]
If $\Omega$ is a $\mathbb{C}$-proper convex domain in $\mathbb{C}^n (n\geq 2)$, then the Kobayashi metric $K_{\Omega}$ is complete.
\end{proposition}
Here $\mathbb{C}$-{\it proper} means that $\Omega$ does not contain any entire complex affine lines. Since all bounded domains are $\mathbb{C}$-proper, thus for a bounded convex domain $\Omega$, the Kobayashi metric $K_\Omega$ on $\Omega$ is a complete length metric. Therefore, $(\Omega,K_{\Omega})$ is a geodesic space.

\subsection{Rough quasi-isometries}\

Suppose that $(X,\:\rho)$ is a metric space and $I\subset\mathbb{R}$ is an
interval. A map $\sigma: I\rightarrow X$ is called a {\it geodesic}
if for all $s,\: t\in I$,
$$\rho(\sigma(s),\sigma(t))=|t-s|.$$
For $\lambda \geq 1$ and $\kappa \geq 0$, a curve $\sigma : I
\rightarrow X$ is called a $(\lambda, \kappa)$-{\it quasi-geodesic}, if for all $s, \:t \in I$,
$$\frac{1}{\lambda}|t-s|-\kappa \leq \rho(\sigma(s), \sigma(t)) \leq
 \lambda|t-s|+\kappa.$$
In particular if $\kappa=0$, it's called a $(\lambda,0)$-quasi-geodesic or $\lambda$-quasi-geodesic.

\bigskip
With the notation as above, for later use we recall the following definition. See e.g. \cite{Schramm1999Embeddings}.
\begin{defn}
Let $f: X\to Y$ be a map between metric spaces.
\begin{enumerate}
\item
If for all $x,\:y\in X$,
$$\frac{d_X(x, \:y)}{\lambda}-\kappa\leq d_Y(f(x),f(y))\leq \lambda d_X(x,\:y)+\kappa,$$
then $f$ is called a {\it $(\lambda, \kappa)$-roughly quasi-isometric map}.
If $\lambda=1$, then $f$ is called a {\it $\kappa$-roughly isometric}.
\noindent
\item Moreover, if $f$ is a homeomorphism and $\kappa=0$, then it is called a {\it $\lambda$-bilipschitz} or simply a {\it bilipschitz}.
\end{enumerate}
\end{defn}
\subsection{Uniformly squeezing properties}\label{usq}\

Following Liu, Sun and Yau \cite{Liu2004Canonical,Kefeng}, a domain $\Omega\subset \mathbb{C}^n$ with $n\geq 2$ is said to be {\it holomorphic homogeneous regular} (HHR) or {\it uniformly squeezing} (USq), if there exists $s>0$ with the following property: for every $z\in\Omega$ there exists a holomorphic embedding $\phi : \Omega \rightarrow \mathbb{C}^n$ with $\phi(z) = 0$ and
$$B_{s}(0)\subset \phi(\Omega)\subset B_1(0),$$
where $B_{1}(0)\subset \mathbb{C}^n$ is the unit ball.

We mention some examples of HHR/USq domains:

(1) $T_{g,n}$, the Teichm\"{u}ller space of hyperbolic surfaces with genus $g>1$ and $n$ punctures;

(2) bounded convex domains \cite{KimON};

(3) strongly pseudoconvex domains \cite{DiederichExposing,squeezingfunc}.

It was shown in \cite{Liu2004Canonical,Yeung2009Geometry,Kefeng} that on an HHR/USq domain $\Omega$, the Carath\'{e}odory metric, the Kobayashi metric, the Bergman metric and the K\"{a}hler-Einstein metric are bilipschitzly equivalent to each other.

\subsection{Gromov products and Gromov hyperbolicities}\label{GGG}\

\begin{defn}
Let $(X, \:\rho)$ be a metric space. Given three points $x, \:y, \:o \in$ $X,$ the {\it Gromov product} of $x,y$ with respect to $o$ is given by
$$(x | y)_{o}=\frac{1}{2}\Big(\rho(x, o)+\rho(o, y)-\rho(x, y)\Big).$$
A proper geodesic metric space $(X, \:\rho)$ is called {\it Gromov hyperbolic} (or $\delta$-{hyperbolic}), if there exists $\delta \geq 0$ such that, for all $o, \:x, \:y, \:z
\in X$,
$$(x | y)_{o} \geq \min \left\{(x | z)_{o},\: (z | y)_{o}\right\}-\delta.$$

\end{defn}
By the triangle inequality, we know that
$$(x | y)_{o}\leq \rho(o,\: [x,y]),$$
where $[x,y]$ is a geodesic connecting $x$ and $y$ in $(X,\: \rho)$. Moreover, if $X$ is Gromov hyperbolic, then we have the following standard estimate
\begin{align}\label{w1}
|(x|y)_{o}-\rho(o,\: [x,y])|\leq \delta'
\end{align}
for some $\delta'>0$.

Note that the large scaled behavior of quasi-geodesics in a Gromov hyperbolic space mimics that of
geodesics rather closely.
\begin{thm}$($Stability of quasi-geodesics, Page 401,\cite{BMHA}$)$\label{sta}
For all $\delta> 0, \: \lambda \geq1, \: \epsilon>0$, there
exists a constant $R=R(\delta,\lambda,\epsilon)$ with the following property:

If $X$ is a $\delta$-hyperbolic geodesic space, $\gamma$ is a $(\lambda,\: \epsilon)$-quasi-geodesic in $X$ and $[x, y]$
is a geodesic segment joining the endpoints of $\gamma$, then the Hausdorff distance between
$[x, y]$ and the image of $\gamma$ is no more than $R$.
\end{thm}

We now recall the following definition. Refer to \cite{Schramm1999Embeddings, BMHA}.
\begin{defn}
Suppose that $X$ is Gromov hyperbolic.
\begin{enumerate}
\item
A sequence $\{x_i\}$ in $X$ is called a {\it Gromov sequence} if $(x_i|x_j)_o\rightarrow \infty$ as $i,$ $j\rightarrow \infty.$
\item
Two such sequences $\{x_i\}$ and $\{y_j\}$ are said to be {\it equivalent} if $(x_i|y_i)_o\rightarrow \infty$ as $i\to\infty$.
\item
The {\it Gromov boundary} $\partial_G X$ of $X$ is defined to be the set of all equivalence classes of Gromov sequences, and $\overline{X}^G=X \cup \partial_G X$ is called the {\it Gromov closure} of $X$.
\item
For $a\in X$ and $b\in \partial_G X$, the Gromov product $(a|b)_o$ is defined by
$$(a|b)_o= \inf \big\{ \liminf_{i\rightarrow \infty}(a|b_i)_o:\; \{b_i\}\in b\big\}.$$
\item
For $a,\: b\in \partial_G X$, the Gromov product $(a|b)_o$ is defined by
$$(a|b)_o= \inf \big\{ \liminf_{i\rightarrow \infty}(a_i|b_i)_o:\; \{a_i\}\in a\;\;{\rm and}\;\; \{b_i\}\in b\big\}.$$
\end{enumerate}
\end{defn}

The following result states that Gromov hyperbolicity is preserved under rough quasi-isometries.
\begin{thm}$($Page 402,\cite{BMHA}$)$\label{gro2}
Let $X,\:X'$ be geodesic metric spaces and $f:X\rightarrow X'$ be a rough quasi-isometry. If $X$ is Gromov hyperbolic, then $X'$ is also Gromov hyperbolic.
\end{thm}

Finally, for our later use we introduce the following result.

\begin{proposition}\label{z0}$($Lemma $5.11$, \cite{Vai-0}$)$
Let $o, \: z\in X$ and let $X$ be a $\delta$-hyperbolic space, and $\xi,\: \xi'\in\partial_G X$. Then for any sequences $\{y_i\}\in \xi$, $\{y_i'\}\in \xi'$, we have
$$(\xi|\xi')_o\leq \liminf_{i\rightarrow \infty} (y_i|y_i')_o \leq \limsup_{i\rightarrow \infty} (y_i|y_i')_o\leq (\xi|\xi')_o+2\delta.$$
\end{proposition}

\bigskip
\section{\noindent{{\bf $M$-convex domains }}}\label{aa}
In this section we will investigate the geometric properties of the Kobayashi geodesics in $m$-convex domains. We first give the proof of Lemma \ref{est14} by using the idea from \cite{Gehring1979Uniform,Bonk2001Uniformizing}, the differences between the Kobayashi metric and the quasihyperbolic metric necessitate some changes in the proof.

%Firstly we prove a lemma similar with the 'Separation property', which means that whenever $[x,y]$ is a geodesic in $(\Omega,K_{\Omega})$ , $z\in[x,y]$ and $\gamma$ is a curve in $\Omega$ connecting $[x,z)$ and $(z,y]$, then $$B(z,a\delta_{\Omega}(z))\cap\gamma\neq\emptyset$$
%for some $a>0$.
%We follow the idea in \cite{Gehring1979Uniform}\cite{Bonk2001Uniformizing} and give a similar result, the differences between the Kobayashi metric and the quasihyperbolic metric necessitate some changes in the proof.

In order to prove Lemma \ref{est14}, we only need to verify the following result.

\begin{lemma}\label{uniform} Suppose that $\Omega$ is a bounded $m$-convex domain in $\mathbb{C}^n (n\geq 2)$ with Dini-smooth boundary, and suppose that $\gamma\subset \Omega$ is a $\lambda$-quasi-geodesic in the Kobayashi metric $K_\Omega$ connecting $y_1, \:y_2\in\Omega$ with $\lambda\geq 1$. Then for any $\alpha>3m^2-2m$, there exists a constant $\tilde{C}>0$ such that, for every $z=\gamma(t)\in\gamma$,
\begin{align*}
\delta_{\Omega}(z)\geq \tilde{C}\left(l_{d}(\gamma|[0,t])\wedge l_{d}(\gamma|[t,1])\right)^{\alpha}.
\end{align*}
\end{lemma}
\begin{rmk}
This result tells us that, for any curve $\gamma'$ connecting $\gamma|[0,t]$ and $\gamma|[t,1]$, we always have
$$B(z,\delta_{\Omega}^{\frac{1}{\alpha}}(z)/\tilde{C})\cap\gamma'\neq\emptyset.$$
\end{rmk}
\begin{proof}
Put $D=\max\limits_{z\in\gamma} \delta_{\Omega}(z)$. For $i=1,2$, let $N_i$ denote the unique integer such that
$$\frac{D}{2^{N_i+1}}\leq \delta_{\Omega}(y_i)\leq \frac{D}{2^{N_i}}.$$
For $k=0,\ldots,N_1$, let $x_k^1$ be the first point on $\gamma$ with
$$\delta_{\Omega}(x_k^1)=\frac{D}{2^k}$$ when a point travels from $y_1$ towards $y_2$. Then we can similarly define $x_k^2$ for $k=0,...,N_2$ with travel direction from $y_2$ to $y_1$.
By using points $x_k^1$ and $x_k^2$ together with the end points $y_1$ and $y_2$, we can divide $\gamma$ into $(N_1+N_2 +3)$ nonoverlapping (modulo end points) subcurves $\gamma_{\nu}$, $\nu\in[-N_{1}-1,N_2+1]$. Note that a curve containing one end point of $\gamma$, as well as the middle subcurve between $x_0^1$ and $x_0^2$, may degenerate. All subcurves $\gamma_{\nu}$ are Kobayashi $\lambda$-quasi-geodesics between their repective end points, and
\begin{align}\label{est8}
&\delta_{\Omega}(u)\leq \frac{D}{2^{|\nu|-1}}, \text{ if } u\in \gamma_{\nu},\notag\\
&\delta_{\Omega}(u)\geq\frac{D}{2^{|\nu|}}, \text{ if } u \text{ is one end point of } \gamma_{\nu}.
\end{align}
It thus follows from (\ref{est}),(\ref{est8}) and the definition of the quasi-geodesic that there exists a constant $A>2$ such that
\begin{align}\label{ineq}
 l_{k}(\gamma_{\nu})\leq \lambda\log\left(1+A\frac{2^{|\nu|}}{D}l_{d}(\gamma_{\nu})\right).
\end{align}
And by (\ref{def}) and (\ref{est5}), we have
\begin{align*}
l_{k}(\gamma_{\nu})\geq\frac{l_{d}(\gamma_{\nu})}{2C\left(\frac{D}{2^{|\nu|-1}}\right)^{\frac{1}{m}}}=
\frac{2^{\frac{|\nu|-1}{m}}}{2CD^{\frac{1}{m}}}l_{d}(\gamma_{\nu}),
\end{align*}
where $C$ is the constant from Definition \ref{def2}.
It's easy to see that, for any $N\in\mathbb{N}$ there exists $C(N)>0$ such that
$$\log(1+x)\leq C(N)x^{1/N},$$
for $x\geq 0$.
Then for all $N>m$, clearly
\begin{align}\label{est13}
\frac{2^{\frac{|\nu|-1}{m}}}{2CD^{\frac{1}{m}}}l_{d}(\gamma_{\nu})\leq l_{k}(\gamma_{\nu})\leq \lambda C(N)A^{\frac{1}{N}}\frac{2^{\frac{|\nu|}{N}}}{D^{\frac{1}{N}}}l_{d}^{\frac{1}{N}}(\gamma_{\nu}),
\end{align}
which implies that
\begin{align}\label{ineq2}
l_{d}(\gamma_{\nu})\leq C' \left(\frac{D}{2^{|\nu|}}\right)^{\frac{\frac{1}{m}-\frac{1}{N}}{1-\frac{1}{N}}},
\end{align}
where $C'=\left(2^{1+\frac{1}{m}}\lambda C(N)A^{\frac{1}{N}}C\right)^{\frac{1}{1-1/N}}$. Suppose $z=\gamma(t)\in\gamma_{\nu}$ for some $\nu$. This yields that
\begin{align}\label{est7}
l_d(\gamma|[0,t])\wedge l_d(\gamma|[t,1])\leq C'\sum\limits_{j\geq |\nu|}\left(\frac{D}{2^{|\nu|}}\right)^{\frac{\frac{1}{m}-\frac{1}{N}}{1-\frac{1}{N}}}\leq 2C'\left(\frac{D}{2^{|\nu|}}\right)^{\frac{\frac{1}{m}-\frac{1}{N}}{1-\frac{1}{N}}}.
\end{align}
Moreover, by formulas (\ref{ineq}) and (\ref{ineq2}), we obtain
\begin{align*}
l_{k}(\gamma_{\nu})\leq \log\left(1+2AC'\left(\frac{D}{2^{|\nu|}}\right)^{\frac{\frac{1}{m}-1}{1-\frac{1}{N}}}\right).
\end{align*}
Therefore, we only need to estimate $\delta_{\Omega}(z)$.
Let $x$ be one end point of $\gamma_{\nu}$. By the estimate ($\ref{est4}$), we conclude that
\begin{align*}
K_{\Omega}\left(x, \:z\right) \geq \frac{1}{2}\left|\log \frac{d\left(x,H\right)}{d\left(z,H\right)}\right|
\end{align*}
holds for any complex hyperplane $H \subset \mathbb{C}^{n}$ with $H \cap \Omega=\emptyset$.
%where $v=x-z\in\mathbb{C}^n$ and $t\vee s:=\max\{t,s\}$ for $t,s\geq 0$.
Therefore, it follows from Definition \ref{def2} that
$$\frac{1}{2}\log\frac{\delta_{\Omega}(x)}{C\delta^{\frac{1}{m}}_{\Omega}(z)}\leq K_\Omega(x,z)\leq l_{k}(\gamma_{\nu})\leq \log\left(1+2AC'\left(\frac{D}{2^{|\nu|}}\right)^{\frac{\frac{1}{m}-1}{1-\frac{1}{N}}}\right).$$
This guarantees that
\begin{align*}
\frac{\delta_{\Omega}(x)}{C\delta^{\frac{1}{m}}_{\Omega}(z)}\leq \left(1+2AC'\left(\frac{D}{2^{|\nu|}}\right)^{\frac{\frac{1}{m}-1}{1-\frac{1}{N}}}\right)^{2}
\end{align*}
and
\begin{align*}
\delta_{\Omega}(z)&\geq \frac{\delta^{m}_{\Omega}(x)}{C''\left(\frac{D}{2^{|\nu|}}\right)^{\frac{\frac{1}{m}-1}{1-\frac{1}{N}}\cdot 2m}}\\
&\geq \frac{1}{C''}\left(\frac{D}{2^{|\nu|}}\right)^{m-\frac{2-2m}{1-1/N}},
\end{align*}
where $C''=(4AC')^{2m}C^{m}$.
Hence by (\ref{est7}) we deduce that
\begin{align}\label{est15}
\delta_{\Omega}(z)\geq \tilde{C}\left(l_d(\gamma|[0,t])\wedge l_d(\gamma|[t,1])\right)^{\alpha},
\end{align}
where
$$\alpha=\frac{3m-2-\frac{m}{N}}{\frac{1}{m}-\frac{1}{N}}\;\;\;\;\mbox{and}\;\;\;\;\tilde{C}=\frac{{C''}}{2C'}.$$

This implies that $\tilde{C}$ depends only on $m$ and $C$ defined in (\ref{def}), and $\alpha$ depends only on $N$ and $m$.
By taking $N\rightarrow \infty$, we have
$\alpha\rightarrow 3m^2-2m$.
Note that, if we increase $N$, then $\tilde{C}$ will become smaller. Therefore, for any $\alpha>3m^2-2m$ there exists $\tilde{C}$ such that (\ref{est15}) holds.
This completes the proof.
\end{proof}

%\begin{lemma}[\cite{FedererCurvature,balogh2000gromov}]\label{c2}
%Let $\Omega$ be a bounded domain with $C^2$ smooth boundary. Then there exists a constant $\epsilon_{0}>0$ such that
%
%1. for every point $z\in N_{\epsilon_0}(\partial\Omega)$ with $N_{\epsilon_0}(\partial\Omega)$ the $\epsilon_0$-neighborhood of $\partial\Omega$, there exists a unique point $\pi(z)\in\partial\Omega$ such that $$|z-\pi(z)|=\delta_{\Omega}(z).$$
%
%2. for the fibers of the map $\pi:N_{\epsilon_0}(\partial\Omega)\rightarrow\partial\Omega$, we have
%$$\pi^{-1}(p)=(p-\epsilon_0 n(p),p+\epsilon_0 n(p)),$$
%where $n(p)$ is the outer unit normal vector of $\partial\Omega$ at $p\in\partial\Omega$.
%\end{lemma}
%By Lemma \ref{c2} and estimate (\ref{est4}), we have the following estimate:
The following result is a direct consequence of the estimation (\ref{est4}) from Subsection \ref{km}.
\begin{lemma}\label{est100}
Let $\Omega$ be a bounded convex domain in $\mathbb{C}^n$ with $n\geq 2$ and $\omega_0\in\Omega$. Then there exists $K>0$ such that the Kobayashi metric
\begin{align}\label{est3}
K_{\Omega}(z,\: \omega_0)\geq\frac{1}{2}\log\frac{1}{\delta_{\Omega}(z)}-K.
\end{align}
\end{lemma}

%Fix $\omega_0\in\Omega$ and take $\epsilon_0$ as in the last Lemma. Then for any $z\in N_{\epsilon_0}(\partial\Omega)$. Let $z_0=z-\epsilon_0\vec{n}(\pi(z))$.
%\begin{align*}
%K_{\Omega}(\omega,z)&\geq K_{\Omega}(z,z_0)-K_{\Omega}(\omega,z_0)\\
%&\geq\frac{1}{2}\log\frac{\epsilon_0}{\delta_{\Omega}(z)}-K_{\Omega}(\omega,z_0)\\
%&\geq\frac{1}{2}\log\frac{1}{\delta_{\Omega}(z)}-K.
%\end{align*}
%for some $K>0.$

%\begin{thm}\label{mcon}
%Let $\Omega$ is a bounded m-convex domain with $C^2$ smooth boundary. Then for any $c_2<\frac{1}{12m^2-8m}$, there exists $c_1>0$ such that $\forall x,y\in\Omega$ ,
%$$l_{d}([x,y])\leq c_1|x-y|^{c_2}.$$
%Moreover, if $(\Omega,K_{\Omega})$ is Gromov hyperbolic, then for any $c_2<\frac{1}{12m^2-8m}$ and $\lambda>1$ there exists $c_1>0$ such that $\forall x,y\in\Omega$
%$$l_{d}(\gamma)\leq c_1|x-y|^{c_2}.$$
%where $\gamma$ is a $\lambda$-quasi-geodesic joining $x$ and $y$.
%\end{thm}

\bigskip
By using Lemma \ref{uniform}, we are now in a position to prove Theorem \ref{thm}.

\textbf{Proof of Theorem \ref{thm}.} Without loss of generality we may assume that diam$(\Omega)\leq 1$ (by scaling $\Omega$ if necessary).
 Fix $\omega\in\Omega$. By (\ref{est}) and (\ref{est3}), it follows immediately that the Gromov product $(x|y)_{\omega}$ satisfies:
\begin{align*}
2(x|y)_{\omega}&=K_{\Omega}(x,\omega)+K_{\Omega}(y,\omega)-K_{\Omega}(x,y)\\
&\geq \frac{1}{2}\log \frac{1}{\delta_{\Omega}(x)}+\frac{1}{2}\log \frac{1}{\delta_{\Omega}(y)}-\frac{1}{2}\log \frac{(\sqrt{\delta_{\Omega}(x)\delta_{\Omega}(y)}+A|x-y|)^2}{\delta_{\Omega}(x)\delta_{\Omega}(y)}-2K\\
&=\frac{1}{2}\log\frac{1}{\delta_{\Omega}(x)\delta_{\Omega}(y)+|x-y|(2A\sqrt{\delta_{\Omega}(x)\delta_{\Omega}(y)}+A^2 |x-y|)}-2K\\
&\geq \frac{1}{2}\log\frac{1}{\delta_{\Omega}(x)\delta_{\Omega}(y)+(A^2+2A)|x-y|}-2K,
\end{align*}
where $A$ is the constant from Lemma \ref{nik} such that (\ref{est}) holds.

In order to estimate the Euclidean length of the geodesic $[x,y]$, there are two cases to consider:

\textbf{Case a}: $|x-y|\geq (\delta_{\Omega}(x)\delta_{\Omega}(y))^2$.
Hence
\begin{align*}
(x|y)_{\omega}&\geq \frac{1}{4}\log \frac{1}{(A+1)^2|x-y|^{\frac{1}{2}}}-K\\
&\geq \frac{1}{8}\log \frac{1}{|x-y|}-K',
\end{align*}
where $K'=K-\frac{1}{4}\log \frac{1}{(A+1)^2}$. By using the definition of the Gromov product, we know that
\begin{align}\label{est19}
K_{\Omega}(\omega,[x,y])\geq(x|y)_{\omega}\geq \frac{1}{8}\log \frac{1}{|x-y|}-K'.
\end{align}
Thus by Lemma \ref{nik}, for any $z\in[x,y]$ we see that there exists $K''>0$ such that
$$\frac{1}{2}\log\frac{1}{\delta_{\Omega}(z)}+K'' \geq K_{\Omega}(\omega,z)\geq \frac{1}{8}\log \frac{1}{|x-y|}-K'.$$
Choosing a point $z\in [x,y]$ with $l_{d}([x,z])=l_d([z,y])=\frac{1}{2}l_{d}([x,y])$,
then by Lemma \ref{uniform} we now have
$$l_{d}([x,y])\leq 2\left(\frac{e^{2K'+2K''}}{\tilde{C}}\right)^{\frac{1}{\alpha}}|x-y|^{\frac{1}{4\alpha}}.$$
Therefore \textbf{Case a} is proved.

\textbf{Case b}: $|x-y|\leq (\delta_{\Omega}(x)\delta_{\Omega}(y))^2$.
It follows from Lemma \ref{est2} that
$$k_{\Omega}(z;v)\geq\frac{|v|}{2\delta_{\Omega}(z)}\geq \frac{|v|}{2},$$
since diam$(\Omega)\leq 1$.
Thus
\begin{align}\label{est20}
\frac{1}{2}l_{d}([x,y])\leq l_{k}([x,y])= K_{\Omega}(x,y)&\leq \log\left(1+A\frac{|x-y|}{\delta_{\Omega}(x)\wedge\delta_{\Omega}(y)}\right)\\
&\leq \log(1+A|x-y|^{\frac{1}{2}})\notag\\
&\leq A|x-y|^{\frac{1}{2}}\notag,
\end{align}
which implies that $l_{d}([x,y])\leq c_1|x-y|^{c_2}$,
where
$$c_1=2\left(\frac{e^{2K'+2K''}}{\tilde{C}}\right)^{\frac{1}{\alpha}}\vee (2A)\;\;\;\;\mbox{and}\;\;\;\; c_2=\frac{1}{4\alpha}<\frac{1}{12m^2-8m}.$$
This completes the first part of the proof.

\bigskip
For the second part, we only need slight modifications of (\ref{est19}) and (\ref{est20}).

Assume that $(\Omega,\: K_{\Omega})$ is Gromov hyperbolic ($\delta$-hyperbolic) and $\gamma$ is a Kobayashi $\lambda$-quasi-geodesic connecting $x,\:y$ with $\lambda\geq 1$. Then it follows from Theorem \ref{sta} that the Kobayashi Hausdorff distance between
$[x,y]$ and the image of $\gamma$ is no more than $R=R(\delta,\lambda)$. Thus, we can take
\begin{align*}
 K_{\Omega}(\omega,\gamma)+R\geq K_{\Omega}(\omega,[x,y])\geq(x|y)_{\omega}\geq \frac{1}{8}\log \frac{1}{|x-y|}-K'
\end{align*}
instead of (\ref{est19}). Also we can take
\begin{align*}
\frac{1}{2\lambda}l_{d}(\gamma)\leq\frac{1}{\lambda}l_{k}(\gamma)\leq l_{k}([x,y])= K_{\Omega}(x,y)&\leq\log\left(1+A\frac{|x-y|}{\delta_{\Omega}(x)\wedge\delta_{\Omega}(y)}\right)
\end{align*}
instead of (\ref{est20}). Furthermore, noting that these changes make no influence on the constant $c_2$, we complete the proof of the second part.
\qed

 \begin{rmk}\label{hh}
Suppose that $\Omega$ is a bounded $m$-convex complex domain. For any two points $x,y\in\Omega$, there exists a complex geodesic which contains $x,y$ in its image. According to a well-known result due to Hardy-Littlewood, any complex geodesic in $\Omega$ extends continuously to its boundary $\partial \Omega$.

Conversely, Mercer \cite{MercerComplex} proved that: for any two points $x,y\in\overline{\Omega}$, there is a complex geodesic whose continuous extension contains $\{x, \:y\}$ in its image. Thus the first part of Theorem \ref{thm} also holds for $x,\:y\in\overline{\Omega}.$
 \end{rmk}

\bigskip
\section{\noindent{{\bf Strongly pseudoconvex domains}}}\label{bb}
In this part we will establish a similar result for strongly pseudoconvex domains. The goal of this section is to prove Theorem \ref{thm-2} and Corollary \ref{cor}.
At first we need some auxiliary results.
%We introduce the function $g:\Omega\times\Omega\rightarrow \mathbb{R}$
%$$g_{\Omega}(x,y)= \log \left[\frac{\left(d_{H}(\pi(x), \pi(y))+\sqrt{\delta_{\Omega}(x)} \vee \sqrt{\delta_{\Omega}(y)}\right)^2}{\sqrt{\delta_{\Omega}(x)\delta_{\Omega}(y)} }\right]$$
%where $\pi:\Omega\rightarrow \partial \Omega$ is the projection mapping and $d_{H}$ is the Carnot-Carath\'{e}odory metric on $\partial\Omega$. Since we only use a rough estimate of $g_{\Omega}$ we don't need knowledge of Carnot-Carath\'{e}odory metric, we will not give the definition. We refer the readers to \cite{balogh2000gromov} for more information. In fact we only use the estimate
%\begin{align}
%g_{\Omega}(x,y)\geq\frac{1}{2}|\log\frac{\delta_{\Omega}(x)}{\delta_{\Omega}(y)}|.
%\end{align}
\begin{lemma}[Lemma 4.1,\cite{balogh2000gromov}]\label{est12}
Let $\Omega$ be a bounded strongly pseudoconvex domain in $\mathbb{C}^n(n\geq 2)$ with $C^2$-smooth boundary. There exists $C>0$ such that for any $x,\: y\in\Omega$,
\begin{align}\label{est21}
K_{\Omega}(x, y) \geq \frac{1}{2}\left|\log\frac{\delta_{\Omega}(x)}{\delta_{\Omega}(y)}\right|-C.
\end{align}
\end{lemma}
The estimate also shows that $(\Omega,\: K_{\Omega})$ is complete and thus it is a geodesic space.
\begin{lemma}[Proposition 1.2,\cite{balogh2000gromov}]\label{est9}
Let $\Omega$ be a bounded strongly pseudoconvex domain in $\mathbb{C}^n(n\geq 2)$ with $C^2$-smooth boundary, Then, for every $\epsilon>0$, there exists $\epsilon_{0}>0$ and $C \geq 0$ such that, for all $z \in N_{\epsilon_{0}}(\partial \Omega) \cap \Omega$ and all $v \in \mathbb{C}^{n}$,
\begin{align}\label{est10}
\begin{aligned}
\left(1-C \delta_\Omega^{1 / 2}(z)\right) &\left(\frac{\left|v_{N}\right|^{2}}{4 \delta_\Omega^{2}(z)}+(1-\epsilon) \frac{L_{\rho}\left(\pi(z) ; v_{H}\right)}{\delta_\Omega(z)}\right)^{1 / 2} \leq K(z ; v) \\
& \leq\left(1+C \delta_\Omega^{1 / 2}(z)\right)\left(\frac{\left|v_{N}\right|^{2}}{4 \delta_\Omega^{2}(z)}+(1+\epsilon) \frac{L_{\rho}\left(\pi(z) ; v_{H}\right)}{\delta_\Omega(z)}\right)^{1 / 2}.
\end{aligned}
\end{align}
\end{lemma}

By using Lemma \ref{est9}, we immediately obtain that, for some $C_1>0$,
\begin{align}\label{est11}
k_{\Omega}(z;v)\geq C_1\frac{|v|}{\delta_{\Omega}^{\frac{1}{2}}(z)}.
\end{align}
In fact, if $z\in N_{\epsilon_{0}}(\partial \Omega) \cap \Omega$, then (\ref{est10}) obviously implies (\ref{est11}). For those $z\in\Omega \backslash N_{\epsilon_{0}}(\partial \Omega)$, we have
$$k_{\Omega}\left(z,\: \frac{v}{|v|}\right)\geq\delta_0>0$$
for some $\delta_0>0$ and $\delta_{\Omega}(z)>\epsilon_0$. Thus we obtain the desired inequality (\ref{est11}) for some $C_1>0$.

\bigskip
The following result is similar to Lemma \ref{uniform}, which can be viewed as the (separation property) geometric characteristic of the Kobayashi (quasi-) geodesic.

\begin{lemma}\label{sep} Suppose that $\Omega$ is a bounded stronly pseudoconvex domain in $\mathbb{C}^n(n\geq 2)$ with $C^2$-smooth boundary. And suppose that $\gamma$ is a Kobayashi $\lambda$-quasi-geodesic joining $y_1$ and $y_2$ with $\lambda\geq 1$. Then for any $\alpha>4$, there exists a constant $\tilde{C}>0$ such that, for every $z=\gamma(t)\in\gamma$,
\begin{align}
\delta_{\Omega}(z)\geq \tilde{C}(l_{d}(\gamma|[0,t])\wedge l_{d}(\gamma|[t,1]))^{\alpha}.
\end{align}
%where $\alpha=(3m-\frac{2}{m})m^2$ and $\tilde{C}$ only denpends on $C$ and $m$ in Definition \ref{def2}.
\end{lemma}
\begin{proof}
Put $$D=\max\limits_{z\in \gamma} \delta_{\Omega}(z).$$
For $i=1,2$, let $N_i$ be the unique integer such that
$$\frac{D}{2^{N_i+1}}< \delta_{\Omega}(y_i)\leq \frac{D}{2^{N_i}}.$$
Define $x_k^1,x_k^2$ and $\gamma_{\nu}$ as in the proof of Lemma \ref{uniform}
%For $k=0,...,N_1$, let $x_k^1$ be the first point on $\gamma$ with
%$$\delta_{\Omega}(x_k^1)=\frac{D}{2^k}$$ when a point travels from $y_1$ towards $y_2$.
%Then one can similarly define $x_k^2$ for $k=0,...,N_2$ with travel direction from $y_2$ to $y_1$.
% We divide $\gamma$ into $(N_1+N_2 +3)$ nonoverlapping subcurves $\gamma_{\nu}$, $\nu\in[-N_{1}-1,N_2+1]$. All subcurves $\gamma_{\nu}$ are Kobayashi $\lambda$-quasi-geodesics between their respective end points, and
and
\begin{align}\label{est17}
&\delta_{\Omega}(u)\leq \frac{D}{2^{|\nu|-1}}, \text{ if } u\in \gamma_{\nu},\notag\\
&\delta_{\Omega}(u)\geq\frac{D}{2^{|\nu|}}, \text{ if } u \text{ is one end point of } \gamma_{\nu}.
\end{align}
It thus follows from (\ref{est}) and (\ref{est17}) and the definition of $\lambda$-quasi-geodesic that there exists a constant $A>0$ such that
\begin{align}\label{ineq3}
 l_{k}(\gamma_{\nu})\leq \lambda \log\left(1+A\frac{2^{|\nu|}}{D}l_{d}(\gamma_{\nu})\right),
\end{align}
where $A$ is the constant given by Lemma \ref{nik} such that (\ref{est}) holds.

By the estimate (\ref{est11}), we have
\begin{align*}
l_{k}(\gamma_{\nu})\geq\frac{C_1l_{d}(\gamma_{\nu})}{\left(\frac{D}{2^{|\nu|-1}}\right)^{\frac{1}{2}}}=
\frac{C_12^{\frac{|\nu|-1}{2}}}{D^{\frac{1}{2}}}l_{d}(\gamma_{\nu}).
\end{align*}
It's easy to see that for any $N\in\mathbb{N}$, there exists $C(N)>0$ such that
$$\log(1+x)\leq C(N)x^{\frac{1}{N}},$$
for $x\geq 0$.
Thus, if we take $N>2$, we have
\begin{align*}
C_1\frac{2^{\frac{|\nu|-1}{2}}}{D^{\frac{1}{2}}}l_{d}(\gamma_{\nu})\leq l_{k}(\gamma_{\nu})\leq \lambda C(N)A^{\frac{1}{N}}\frac{2^{\frac{|\nu|}{N}}}{D^{\frac{1}{N}}}l_{d}^{\frac{1}{N}}(\gamma_{\nu}),
\end{align*}
which implies that
\begin{align}\label{ineq4}
l_{d}(\gamma_{\nu})\leq C' \left(\frac{D}{2^{|\nu|}}\right)^{\frac{\frac{1}{2}-\frac{1}{N}}{1-\frac{1}{N}}},
\end{align}
where $C'=(2^{\frac{1}{2}}\lambda C(N)A^{\frac{1}{N}}/C_1)^{\frac{N}{N-1}}$.

Assume that $z=\gamma(t)\in\gamma_{\nu}$ for some index $\nu$.  Therefore, we obtain
\begin{align}\label{est18}
l_d(\gamma|[0,t])\wedge l_d(\gamma|[t,1])\leq C'\sum\limits_{j\geq |\nu|}\left(\frac{D}{2^{j}}\right)^{\frac{\frac{1}{2}-\frac{1}{N}}{1-\frac{1}{N}}}\leq 2C'\left(\frac{D}{2^{|\nu|}}\right)^{\frac{\frac{1}{2}-\frac{1}{N}}{1-\frac{1}{N}}}.
\end{align}
Moreover, in light of formulae (\ref{ineq3}) and (\ref{ineq4}), we get
\begin{align*}
l_{k}(\gamma_{\nu})\leq \log\left(1+2AC'\left(\frac{D}{2^{|\nu|}}\right)^{\frac{-\frac{1}{2}}{1-\frac{1}{N}}}\right).
\end{align*}

Suppose $x$ is one end point of $\gamma_{\nu}$. The estimation $\delta_{\Omega}(z)$ consists of two cases. With the constant $C$ given by Lemma \ref{est12}, we distinguish two cases:

\textbf{Case I}: $\frac{1}{2}|\log\frac{\delta_{\Omega}(x)}{\delta_{\Omega}(z)}|\leq NC$.
Hence
$$\delta_{\Omega}(z)\geq e^{-2NC}\delta_{\Omega}(x)\geq e^{-2NC}\frac{D}{2^{|\nu|}}.$$

\textbf{Case II}: $\frac{1}{2}|\log\frac{\delta_{\Omega}(x)}{\delta_{\Omega}(z)}|> NC$.
By Lemma \ref{est12}, we obtain
$$K_{\Omega}(x,z)\geq \frac{1}{2}\left|\log\frac{\delta_{\Omega}(x)}{\delta_{\Omega}(z)}\right|-C\geq \frac{N-1}{2N}\left|\log\frac{\delta_{\Omega}(x)}{\delta_{\Omega}(z)}\right|,$$
which implies that
$$\frac{N-1}{2N}\log\frac{\delta_{\Omega}(x)}{\delta_{\Omega}(z)}\leq K_{\Omega}(x,z)\leq l_{k}(\gamma_{\nu})\leq \log\left(1+2AC'\left(\frac{D}{2^{|\nu|}}\right)^{\frac{-\frac{1}{2}}{1-\frac{1}{N}}}\right).$$
Thus it follows that
\begin{align*}
\frac{\delta_{\Omega}(x)}{\delta_{\Omega}(z)}\leq \left(1+2AC'\left(\frac{D}{2^{|\nu|}}\right)^{\frac{-\frac{1}{2}}{1-\frac{1}{N}}}\right)^{\frac{2N}{N-1}}
\end{align*}
and
\begin{align*}
\delta_{\Omega}(z)\geq \frac{1}{C''}\left(\frac{D}{2^{|\nu|}}\right)^{1+\frac{N^2}{(N-1)^2}},
\end{align*}
where $C''=(4AC')^{\frac{2N}{N-1}}$.
Now (\ref{est18}) implies that
\begin{align}\label{est16}
\delta_{\Omega}(z)\geq \tilde{C}\left(l_d(\gamma|[0,t])\wedge l_d(\gamma|[t,1])\right)^{\alpha},
\end{align}
where
$$\tilde{C}=\frac{{C''}\wedge e^{-2NC}}{C'}\;\;\;\mbox{and}\;\;\; \alpha=\left(1+\frac{N^2}{(N-1)^2}\right)\cdot\frac{1-\frac{1}{N}}{\frac{1}{2}-\frac{1}{N}}.$$

Noting that $\lim\limits_{N\rightarrow \infty}\alpha=4$, thus for any $\alpha>4$ there exists $\tilde{C}>0$ such that (\ref{est16}) holds.
This completes the proof.
\end{proof}

We remark that the proof of Theorem \ref{thm-2} follows almost the same line as the proof of Theorem \ref{thm}.
We only need a slight modification for the estimate of the Kobayashi metric. For
the sake of completeness, we present its simple proof here.
%
%The proof of Lemma \ref{sep} and the following Theorem is essentially the same with the proof of Lemma \ref{uniform} and Theorem \ref{mcon}, only the estimates of the Kobayashi metric necessitates some changes.

%\begin{thm}\label{str}
%Let $\Omega$ is a bounded strongly pseudoconvex domain with $C^2$ smooth boundary. Then for any $c_2<\frac{1}{16}$ and $\lambda>1$ there exists $c_1>0$ such that $\forall x,y\in\Omega$,
%$$l_{d}(\gamma)\leq c_1|x-y|^{c_2}$$
%where $\gamma$ is a $\lambda$-quasi-geodesic joining $x$ and $y$.
%\end{thm}

\textbf{Proof of Theorem \ref{thm-2}}. Scaling domain as necessary, we may assume without loss of generality that diam$(\Omega)\leq 1$.

Fix $\omega\in\Omega$. By (\ref{est21}), we have
$$K_{\Omega}(z,\omega)\geq \frac{1}{2}\log \frac{1}{\delta_{\Omega}(z)}-K,$$
for some $K>0$. Then by using (\ref{est}), we deduce that the Gromov product $(x|y)_{\omega}$ satisfies
\begin{align*}
2(x|y)_{\omega}&=K_{\Omega}(x,\omega)+K_{\Omega}(y,\omega)-K_{\Omega}(x,y)\\
%&\geq \frac{1}{2}\log \frac{1}{\delta_{\Omega}(x)}+\frac{1}{2}\log \frac{1}{\delta_{\Omega}(y)}-\frac{1}{2}\log \frac{(\sqrt{\delta_{\Omega}(x)\delta_{\Omega}(y)}+A|x-y|)^2}{\delta_{\Omega}(x)\delta_{\Omega}(y)}-2K\\
%&=\frac{1}{2}\log\frac{1}{\delta_{\Omega}(x)\delta_{\Omega}(y)+|x-y|(2A\sqrt{\delta_{\Omega}(x)\delta_{\Omega}(y)}+A^2 |x-y|)}-2K\\
&\geq \frac{1}{2}\log\frac{1}{\delta_{\Omega}(x)\delta_{\Omega}(y)+(A^2+2A)|x-y|}-2K,
\end{align*}
where $A$ is the constant in Lemma \ref{nik} such that (\ref{est}) holds.

\bigskip
With the notation in Theorem \ref{thm-2}, recall that $\gamma$ is a Kobayashi $\lambda$-quasi-geodesic connecting $x$ and $y$ with $\lambda\geq 1$. To estimate the Euclidean length of $\gamma$, we consider two cases:

\textbf{Case A}: $|x-y|\geq (\delta_{\Omega}(x)\delta_{\Omega}(y))^2$.
By the assumption, it follows that
\begin{align*}
(x|y)_{\omega}&\geq \frac{1}{4}\log \frac{1}{(A+1)^2|x-y|^{\frac{1}{2}}}-K\\
&\geq \frac{1}{8}\log \frac{1}{|x-y|}-K',
\end{align*}
where $K'=K-\frac{1}{4}\log \frac{1}{(A+1)^2}$. Then by the definition of Gromov product, we know that
$$K_{\Omega}(\omega,[x,y])\geq(x|y)_{\omega}\geq \frac{1}{8}\log \frac{1}{|x-y|}-K',$$
where $[x,\:y]$ denotes a Kobayashi geodesic connecting $x$ and $y$ in $\Omega$. Noticing that any bounded strongly pseudoconvex
domain in $C^n (n\geq2)$ with the Kobayashi metric is Gromov
hyperbolic (see e.g. Theorem \ref{cc}), then it follows from Theorem \ref{sta} that, there exists $R>0$ such that
$$K_{\Omega}(\omega,[x,y])\leq K_{\Omega}(\omega,\gamma)+R.$$
Thus by (\ref{est}) for any $z\in\gamma$, there exists $K''>0$ such that
$$\frac{1}{2}\log\frac{1}{\delta_{\Omega}(z)}+K'' \geq K_{\Omega}(\omega,z)\geq \frac{1}{8}\log \frac{1}{|x-y|}-K'-R.$$
Take a point $z=\gamma(t)$ with
$l_{d}(\gamma|[0,t])=l_d(\gamma|[t,1])=\frac{1}{2}l_{d}(\gamma)$.
Therefore, by Lemma \ref{sep}, we get that
$$l_{d}(\gamma)\leq 2\left(\frac{e^{2K'+2K''+2R}}{\tilde{C}}\right)^{\frac{1}{\alpha}}|x-y|^{\frac{1}{4\alpha}}.$$

\textbf{Case B}: $|x-y|\leq (\delta_{\Omega}(x)\delta_{\Omega}(y))^2$.
By the estimate (\ref{est11}) and the fact that diam$(\Omega)\leq1$, we obtain that
$k_{\Omega}(z;v)\geq C_1|v|$ for any $z\in\Omega$ and $v\in\mathbb{C}^n$,
which implies that
\begin{align*}
\frac{C_1}{\lambda}l_{d}(\gamma)\leq \frac{1}{\lambda}l_{k}(\gamma)\leq K_{\Omega}(x,y)&\leq \log\left(1+A\frac{|x-y|}{\delta_{\Omega}(x)\wedge\delta_{\Omega}(y)}\right)\\
&\leq \log(1+A|x-y|^{\frac{1}{2}})\\
&\leq A|x-y|^{\frac{1}{2}}.
\end{align*}
Hence we have $l_{d}(\gamma)\leq c_1|x-y|^{c_2}$,
where
$$c_1=2\left(\frac{e^{2K'+2K''+2R}}{\tilde{C}}\right)^{\frac{1}{\alpha}}\vee \frac{A\lambda}{C_1}\;\;\;\;\mbox{and}\;\;\;\;c_2=\frac{1}{4\alpha},$$
which completes the proof of Theorem \ref{thm-2}.
\qed

%\begin{cor}
%Suppose $\Omega$ is a bounded domain with $C^2$ smooth boundary and satisfies either\\
%(a). $\Omega$ is convex and $(\Omega,K_{\Omega})$ is Gromov hyperbolic or \\
%(b). $\Omega$ is a strongly pseudoconvex domain\\
%Then there exists $c_1,c_2>0$ such that $\forall x,y\in\Omega$,
%$$l_{d}(\gamma)\leq c_1|x-y|^{c_2}$$
%where $\gamma$ is a $\lambda$-quasi-geodesic joining $x$ and $y$. Here $K_{\Omega}$ is one of the Kobayashi metric, Bergman metric, Carath\'{e}odory metric and K\"{a}hler Einstein metric.
%\end{cor}

\bigskip
\textbf{Proof of Corollary \ref{cor}.}

Recall that $\varrho_{\Omega}$ is one of the Kobayashi metric, Bergman metric, Carath\'{e}odory metric and K\"{a}hler-Einstein metric of $\Omega$.

If $\gamma$ is a $\lambda$-quasi-geodesic joining $x$ and $y$ with respect to the metric $\varrho_{\Omega}$, then by using the fact recorded in Section \ref{usq}, we know that $\gamma$ is also a $(C\lambda)$-quasi-geodesic for the Kobayashi metric for some $C>0$.

 On the other hand, by Theorem \ref{gro2}, if $\Omega$ is Gromov hyperbolic with respect to $\varrho_{\Omega}$, then it is also Gromov hyperbolic with respect to the Kobayashi metric. Thus we can complete the proof by using Theorem \ref{thm} and Theorem \ref{thm-2}.
\qed

\bigskip
\section{\noindent{{\bf Applications }}}\label{appp}
We focus our attention in this section to present some applications for our results, i.e., the Gehring-Hayman type theorem (Theorems \ref{thm} and \ref{thm-2}) and the Separation property (Lemmas \ref{est14} and \ref{sep}). At first we prove the bi-H\"{o}lder equivalence between the Euclidean boundary and the Gromov boundary on certain complex domains. Secondly, we use this boundary correspondence to obtain some extension results for biholomorphisms, and more general rough quasi-isometries with respect to the Kobayashi metrics between the domains.

We begin with some necessary definitions and auxiliary results concerning Gromov hyperbolic geometry and morphisms between their boundaries at infinity. After that the proofs of Theorem \ref{app} and Corollary \ref{cor2} are given.

Recall that we have already defined the Gromov product and Gromov hyperbolic spaces in Subection \ref{GGG}. Furthermore, for a Gromov hyperbolic space $X$ one can define a class of {\it visual metrics} on $\partial_G X$ via the extended Gromov products, see \cite{BMHA,Schramm1999Embeddings}. For any metric $\rho_{G}$ in this class, there exist a parameter $\epsilon>0$ and a base point $w \in X$ such that
\begin{align}\label{ee}\rho_{G}(a, b) \asymp e^{ -\epsilon(a|b)_{w}}, \quad \text { for all } a, \:b \in \partial_{G} X.
\end{align}
Here we write $f \asymp g$ for two positive functions if there exists a constant $C \geq 1$ such that $(1 / C) f \leq g \leq C f .$ Any two metrics $d_{1}, \: d_{2}$ in the canonical class are called {\it snowflake equivalent} if and only if the identity map id : $\left(\partial_{G} X, \:d_{1}\right) \rightarrow\left(\partial_{G} X, \: d_{2}\right)$ is a snowflake map. Note that a homeomorphism $\phi:(X_1,\: d_1)\to (X_2, \:d_2)$ between two metric spaces is said to be {\it snowflake} if there exist $\lambda, \:\kappa>0$ such that, for any $x, \:y \in X_1$,
$$(1 / \lambda)d_1(x,\: y)^{\kappa} \leq d_2(\phi(x),\: \phi(y)) \leq \lambda d_1(x,\: y)^{\kappa}.$$

Now we recall the definitions of H\"older and power quasisymmetric mappings as follows. A homeomorphism $\phi:(X_1, \:d_1)\to (X_2,\: d_2)$ between two metric spaces is said to be {\it H\"older} if there exist $\lambda,\: \kappa>0$ such that, for any $x,\: y \in X_1$,
$$d_2(\phi(x),\: \phi(y)) \leq \lambda d_1(x,\: y)^{\kappa}.$$
Moreover, $\phi$ is called {\it bi-H\"older} if there exist $\lambda,\alpha>0$ such that, for any $x,y \in X_1$,
$$(1 / \lambda)d_1(x,\: y)^{1/\alpha} \leq d_2(\phi(x),\: \phi(y)) \leq \lambda d_1(x,\:y)^{\alpha}.$$

\begin{defn}
Let $\phi:(X_1,\: d_1)\to (X_2,\: d_2)$ be a homeomorphism between metric spaces, and $\lambda \geq 1$, $\kappa>0$ be constants.

If for all distinct points $x, \:y, \:z \in X_1$,
$$
\frac{d_2(\phi(x),\:\phi(z))}{d_2(\phi(x),\:\phi(y))} \leq \eta_{\kappa, \lambda}\left(\frac{d_1(x,\:z)}{d_1(x,\:y)}\right),
$$
then $\phi$ is called a $(\lambda,\:\kappa)$-{\it power quasisymmetry}. Here we have used the notation
$$
\eta_{\lambda,\kappa}(t)=\left\{\begin{array}{cl}
\lambda t^{1 / \kappa} & \text { for } 0<t<1, \\
\lambda t^{\kappa} & \text { for } t \geq 1.
\end{array}\right.
$$
\end{defn}

It is easy to see that every snowflake mapping is bi-H\"older. By carefully checking the proof of Theorem 6.15 in \cite{VThe}, we obtain the following proposition.

\begin{proposition}\label{pqs}
A power quasisymmetry between two bounded metric spaces is bi-H\"older.
\end{proposition}

Also, we need an auxiliary result for our later use.
\begin{proposition}[Section 6,\cite{Schramm1999Embeddings}]\label{bonk}
Suppose $f: X\rightarrow Y$ is a rough quasi-isometry between two geodesic Gromov hyperbolic spaces $X$ and $Y$. Then $f$ sends every Gromov sequence in $X$ to a Gromov sequence in $Y$ and induces a power quasisymmetric boundary mapping $\tilde{f}_{G}: \partial_{G} X \rightarrow \partial_{G} Y$, where $\partial_{G} X$ and $ \partial_{G} Y$ are equipped with certain visual metrics.
\end{proposition}

Now we are ready to show Theorem \ref{app}. For the convenience of the reader we restate the theorem as follows.

\begin{thm}
Suppose that $\Omega$ is a bounded domain in $\mathbb{C}^n (n\geq 2)$ and that $\Omega$ satisfies either\\
$($a$)$ $\Omega$ is convex with Dini-smooth boundary and $(\Omega,K_{\Omega})$ is Gromov hyperbolic; or \\
$($b$)$ $\Omega$ is strongly pseudoconvex with $C^2$-smooth boundary. \\
Then the identity map $id: \Omega \rightarrow \Omega$ extends to a bi-H\"{o}lder homeomorphism between the boundaries
\begin{align*}
id:(\partial\Omega,\: |\cdot|)\rightarrow (\partial_{G}\Omega,\: \rho_G)
\end{align*}
(to simplify notation, here use the same notation), where $\rho_G$ belongs to the visual metrics class on the Gromov boundary of $(\Omega,\: K_{\Omega})$ (see (\ref{ee})).
\end{thm}

\begin{proof}
We first record some auxiliary results for later use. On one hand, it follows from Theorems \ref{thm} and \ref{thm-2} that: in both cases (a) and (b) there are constants $c_1, \: c_2>0$ such that, for any $x,\: y\in\Omega$,
\begin{align}\label{z2} l_{d}([x,y])\leq c_1|x-y|^{c_2},\end{align}
where $[x,y]$ is a Kobayashi geodesic joining $x$ and $y$ in $\Omega$.

On the other hand, by using Lemmas \ref{est14} and \ref{sep}, we see that there exist constants ${\tilde{C}},\: \alpha>0$ such that, for every $u\in [x,y]$,
\begin{align}\label{z3}
\delta_{\Omega}(u)\geq \tilde{C}\left(l_{d}([x,u])\wedge l_{d}([u,y])\right)^{\alpha}.
\end{align}

Next we want to show that the identity map extends to a bijection between the Euclidean boundary of $\Omega$ and the Gromov boundary of the space $(\Omega,\: K_\Omega)$. More precisely, a sequence in $\Omega$ is a Gromov sequence if and only if it converges to some boundary point in the Euclidean metric.

To this end, fix a point $\omega \in \Omega$. For any Gromov sequences $\{x_k\},\{y_k\}\subset\Omega$ with $(x_k|y_k)_{\omega}\rightarrow \infty$ as $k\to\infty$. For each $k$, we may connect $x_k$ and $y_k$ by a Kobayashi geodesic $[x_k,\:y_k]$ in $\Omega$. Then
choose a point $u_k\in[x_k,y_k]$ such that
$l_d([x_k,u_k])=l_d([u_k,y_k])$.
Lemma \ref{nik} implies that there exists $K_1>0$ such that
\begin{align}\label{est102}
K_{\Omega}(\omega,\:u_k)\leq\frac{1}{2}\log\frac{1}{\delta_{\Omega}(u_k)}+K_1,
\end{align}
which implies that
$$(x_k|y_k)_{\omega}\leq K_{\Omega}(\omega,\:u_k)\leq \frac{1}{2}\log\frac{1}{\delta_{\Omega}(u_k)}+K_1.$$
Thus we have
$\delta_{\Omega}(u_k)\rightarrow 0$ as $k\to \infty$.
Now by applying (\ref{z3}) to the geodesic $[x_k,y_k]$, we obtain that
\begin{align}\label{est101}
|x_k-y_k|\leq l_{d}([x_k,y_k])\leq \frac{2}{\tilde{C}}\delta_{\Omega}^{\frac{1}{\alpha}}(u_k)\rightarrow 0,\end{align}
as desired.

On the other hand, for every $x\in\partial\Omega$, choose sequences $\{x_k\},\{y_k\}\subset\Omega$ with
$x_k\rightarrow x$ and $y_k\rightarrow x$ as $k\to \infty$. For each $k$, again we may join $x_k$ and $y_k$ by a Kobayashi geodesic $[x_k,y_k]$ in $\Omega$. Choose a point $z_k\in[x_k,y_k]$ such that $$K_{\Omega}(\omega,z_k)=K_{\Omega}(\omega,[x_k,y_k]).$$
Now, applying (\ref{z2}) to the geodesic $[x_k,y_k]$, we conclude that
\begin{align}\label{ff}
\delta_{\Omega}(z_k)&\leq \frac{1}{2}l_d([x_k,y_k])+\delta_{\Omega}(x_k)\wedge\delta_{\Omega}(y_k)\notag\\&\leq
\frac{c_1}{2}|x_k-y_k|^{c_2}+\delta_{\Omega}(x_k)\wedge\delta_{\Omega}(y_k).
\end{align}
By using the estimates (\ref{est3}) and (\ref{est21}) in both cases (a) and (b), it follows immediately that there exists $K_2>0$ such that
\begin{align}\label{zz}
K_{\Omega}(\omega,z_k)\geq\frac{1}{2}\log \frac{1}{\delta_{\Omega}(z_k)}-K_2.\end{align}
Since $|x_k-y_k|\rightarrow 0$ and $\delta_{\Omega}(x_k),\: \delta_{\Omega}(y_k)\rightarrow 0$, it is obvious that
$K_{\Omega}(\omega,[x_k,y_k])\rightarrow\infty$ as $k\to \infty$. Then, the standard estimate (\ref{w1}) implies that
$(x_k|y_k)_\omega\to \infty$ as $k\to\infty$.
Therefore, we have already proved that the identity map $id: \Omega \rightarrow \Omega$ extends to a bijection (still use the same name) $$
id:(\partial\Omega,\: |\cdot|)\rightarrow (\partial_{G}\Omega,\: \rho_G)
$$
between the Euclidean boundary of $\Omega$ and the Gromov boundary of $(\Omega,\: K_\Omega)$.

\bigskip
We are now ready to show the boundary mapping $id:(\partial\Omega,\: |\cdot|)\rightarrow (\partial_{G}\Omega,\: \rho_G)$ is bi-H\"{o}lder continuous, where $\rho_G$ is a visual metric on the Gromov boundary of $(\Omega,\: K_{\Omega})$ with parameter $\varepsilon>0$ and the base point $\omega$ (see (\ref{ee})).

For all $x,\:y\in \partial\Omega$, take two sequences $\{x_k\}$ and $\{y_k\}$ in $\Omega$ such that
$$|x_k-x|\vee |y_k-y|\to 0\;\;\mbox{as}\;\;k\to \infty.$$
By repeating the above argument, we may assume without loss of generality that $x,\: y\in \partial_G\Omega$, and $\{x_k\}\in x$, $\{y_k\}\in y$ are Gromov sequences.

For the one direction, as before let $u_k\in[x_k,y_k]$ such that $l_d([x_k,u_k])=l_d([u_k,y_k])$. Similarly, application of the same argument as above gives (\ref{est102}) and (\ref{est101}). These two estimates guarantee that
\begin{align*}
e^{-(x_k|y_k)_{\omega}}&\geq e^{-K_{\Omega}(\omega,\:u_k)}
\geq e^{-K_1}\delta_{\Omega}^{\frac{1}{2}}(u_k)\\
&\geq \frac{\tilde{C}}{2^{\alpha/2}e^{K_1}}\Big(l_d[x_k,\:y_k]\Big)^{\alpha/2}\geq\frac{\tilde{C}}{2^{\alpha/2}e^{K_1}}|x_k-y_k|^{\alpha/2}.
\end{align*}
Thus it follows from Proposition \ref{z0} and the estimate $(\ref{ee})$ that
\begin{align*}
\rho_G(x,\: y)&\geq\frac{1}{C} e^{-\varepsilon(x|y)_{\omega}}
\geq \frac{1}{C} \limsup_{k\to \infty}e^{-\varepsilon(x_k|y_k)_{\omega}}\\
&\geq \frac{1}{C'} \limsup_{k\to \infty}|x_k-y_k|^{\alpha\varepsilon/2}\\
&=\frac{1}{C'}|x-y|^{\alpha\varepsilon/2}.
\end{align*}

For the other direction, as before we choose a point $z_k\in[x_k,y_k]$ such that $$K_{\Omega}(\omega,\: z_k)=K_{\Omega}(\omega,[x_k,\: y_k]).$$
Again, a similar argument as before shows that the estimates (\ref{ff}) and (\ref{zz}) hold. Moreover, from the standard estimate (\ref{w1}), it is obvious that
\begin{align*}
e^{-(x_k|y_k)_{\omega}}&\asymp e^{-K_{\Omega}(\omega,[x_k,y_k])}\\&=e^{-K_{\Omega}(\omega,z_k)}\leq e^{K_2}\delta_{\Omega}^{\frac{1}{2}}(z_k)\\&\leq e^{K_2}\Big(\frac{c_1}{2}|x_k-y_k|^{c_2}+\delta_{\Omega}(x_k)\vee\delta_{\Omega}(y_k)\Big)^{\frac{1}{2}}.
\end{align*}
Together with Proposition \ref{z0} and the estimate $(\ref{ee})$, we obtain that
\begin{align*}
\rho_G(x,y)&\leq C e^{-\varepsilon(x|y)_{\omega}}
\leq C \liminf_{k\to \infty}e^{-\varepsilon(x_k|y_k)_{\omega}}\\
&\leq C' \liminf_{k\to \infty}\Big(\frac{c_1}{2}|x_k-y_k|^{c_2}+\delta_{\Omega}(x_k)\vee\delta_{\Omega}(y_k)\Big)^{\frac{\varepsilon}{2}}\\
&\leq C''|x-y|^{c_2\varepsilon/2},
\end{align*}
as required.
\end{proof}

%\begin{rmk}\label{rmk1}
%Note that in the second part of the proof, we only need to use Lemma \ref{uniform} which doesn't require any extra boundary regularity assumption. Hence we see that the homeomorphism
%$$id: (\partial_{G}\Omega,\: \rho_G)\rightarrow(\partial\Omega,\: |\cdot|)$$
%is H\"{o}lder continuous without assuming the Dini-smooth boundary condition. However, for the other side, our approach can not remove the Dini-smooth boundary condition.
%\end{rmk}

We also remark that the Gromov hyperbolicity of $(\Omega,\: K_\Omega)$ is only used to get $$e^{-(x|y)_{\omega}}\asymp e^{-K_{\Omega}(\omega,[x,y])}.$$
Thus if we get rid of this condition, by repeating the arguments in the proof of Theorem \ref{app} and using Remark \ref{hh}, we can easily deduce the following byproduct.
\begin{proposition}
Let $\Omega$ be a bounded $m$-convex domain in $\mathbb{C}^n (n\geq 2)$ with Dini-smooth boundary. Then for any $x,y\in\partial\Omega$, we have
\begin{align}\label{gg}
K_{\Omega}(\omega,[x,y])\asymp\log\frac{1}{|x-y|}.
\end{align}
\end{proposition}

Owe to the bilipschitz equivalence of the canonical metrics on $\Omega$, (\ref{gg}) also holds with respect to the Kobayashi metric, the Bergman metric, the Carath\'{e}odory metric and the K\"{a}hler-Einstein metric on $\Omega$. Furthermore, denoting by $\varrho_{\Omega}$ one of the canonical metrics on $\Omega$, we obtain

\begin{cor}\label{w2}
Suppose that $\Omega$ is a bounded domain in $\mathbb{C}^n (n\geq 2)$ and that $\Omega$ satisfies either\\
$($a$)$ $\Omega$ is convex with Dini-smooth boundary and $(\Omega,\: \varrho_{\Omega})$ is Gromov hyperbolic; or \\
$($b$)$ $\Omega$ is strongly pseudoconvex with $C^2$-smooth boundary. \\
Then the identity map $id: \Omega \rightarrow \Omega$ extends to a bi-H\"{o}lder homeomorphism
\begin{align*}
id:(\partial\Omega,\: |\cdot|)\rightarrow (\partial_{G}\Omega,\: \rho),
\end{align*}
where $\rho$ belongs to the visual metrics class on the Gromov boundary of $(\Omega,\: \varrho_{\Omega})$.
\end{cor}

\begin{proof} Note that, by using the fact recorded in Section \ref{usq}, the Kobayashi metric $K_\Omega$ is bilipschitzly equivalent to $\varrho_{\Omega}$ under the identity map. Then by Theorem \ref{gro2}, if $\Omega$ is Gromov hyperbolic with respect to the metric $\varrho_{\Omega}$, then it is also Gromov hyperbolic with respect to the Kobayashi metric $K_{\Omega}$.

From Propositions \ref{bonk} and \ref{pqs}, it follows that the identity map
$$id:(\Omega,\: K_{\Omega})\to (\Omega,\: \varrho_{\Omega})$$
extends to a bi-H\"older homeomorphism between the Gromov boundary of $(\Omega,\; K_{\Omega})$ and the Gromov boundary of $(\Omega,、： \varrho_{\Omega})$ with respect to the visual metrics.

Moreover, it follows from Theorem \ref{app} that the identity map extends to a bi-H\"older homeomorphism between the Euclidean boundary of $\Omega$ and the Gromov boundary of $(\Omega,\: K_{\Omega})$.

Therefore, the conclusion follows easily from these results and the fact that the composition of bi-H\"older mappings is bi-H\"older as well.
\end{proof}

Finally, we conclude this section by proving Corollary \ref{cor2}.

\textbf{Proof of Corollary \ref{cor2}.}
At first, one observes from Propositions \ref{bonk} and \ref{pqs} that
$$f:(\Omega_1,K_{\Omega_1})\to (\Omega_2,K_{\Omega_2})$$
extends to a homeomorphism such that every sequence $\{x_n\}$ in $(\Omega_1,\: K_{\Omega_1})$ is Gromov if and only if the image sequence $\{f(x_n)\}$ in $(\Omega_2,\: K_{\Omega_2})$ is Gromov. Moreover, the induced mapping
$\widetilde{f}:\partial_{G}\Omega_1\to \partial_{G}\Omega_2$
is bi-H\"older when the Gromov boundaries $\partial_{G}\Omega_i$ of $(\Omega_i,K_{\Omega_i})$ are endowed with their visual metrics for $i=1,2$.

For each $i=1,2$, then by Theorem \ref{app}, we see that the identity map
$id_i:\Omega_i\to \Omega_i$
extends as a homeomorphism such that a sequence in $\Omega_i$ is a Gromov sequence if and only if it converges to a point in $\partial \Omega_i$. And the induced mapping
$$id_i:(\partial\Omega_i,\: |\cdot|)\rightarrow (\partial_{G}\Omega_i,\: \rho_i) $$
is bi-H\"older, where $\rho_i$ is a visual metric on the Gromov boundary of $(\Omega_i,K_{\Omega_i})$.

Therefore, we get a well-defined boundary mapping
$$\overline{f}=id_2^{-1}\circ\tilde{f}\circ id_{1}: \partial\Omega_1\to \partial\Omega_2$$
such that $\{x_n\}$ in $\Omega_1$ converges to a point in $\partial \Omega_1$ if and only if $\{f(x_n)\}$ in $\Omega_2$ converges to a point in $\partial \Omega_2$. This shows that $\overline{f}$ is the corresponding continuous extension mapping by $f$, which is a homeomorphism. Consequently, $\overline{f}$ is bi-H\"older because the composition of bi-H\"older mappings is also bi-H\"older.
The proof of the corollary is completed.
\qed

%\begin{rmk}
%If we remove the Dini-smooth condition of $\Omega_2$, and only suppose $\Omega_2$ is a bounded convex domain such that $(\Omega_2,K_{\Omega_2})$ is Gromov hyperbolic, then the induced boundary map $\overline{f}$ is H\"{o}lder continuous by Remark \ref{rmk1}. Moreover, by using a repeated argument as in the proof of Corollary \ref{w2}, similar results hold for Bergman metric, Carath\'{e}odory metric and K\"{a}hler-Einstein metric on $\Omega$.
%\end{rmk}
\bigskip

{\bf Acknowledgement.} The authors would like to thank Professors Manzi Huang and Xiantao Wang for many valuable comments and suggestions.

\bibliography{reference}
\bibliographystyle{plain}{}
\end{document}